\documentclass[11pt]{amsart}
\usepackage{amssymb,amsgen,amsbsy,amsopn,amsfonts,graphicx}
\usepackage{amsmath}
\usepackage[dvips]{color}
\usepackage{multicol}
\usepackage{mathrsfs}
\usepackage{nccmath}
\usepackage{enumerate}
\usepackage{url}
\usepackage{amsthm}
\newtheorem{lemma}{Lemma}[section]
\newtheorem{theorem}[lemma]{Theorem}
\newtheorem{prop}[lemma]{Proposition}

\newtheorem{rem}[lemma]{Remark}

\theoremstyle{definition}

\newtheorem*{MT}{Main Theorem}

\numberwithin{equation}{section}
\numberwithin{lemma}{section}
\def\R{{\mathbb R}}
\newcommand{\ep}{\epsilon}

\newcommand{\p}{\partial}

\newcommand{\om}{\omega}

\begin{document}

\title[Small Breathers]{Small Generalized Breathers with Exponentially
Small Tails for Klein-Gordon Equations}


%

\author[Lu]{Nan Lu}
\address{Department of Mathematics and Statistics \\
University of Massachusetts--Amherst\\ Amherst, MA, 01003}
\email{nanlu@math.umass.edu}
\maketitle

\begin{abstract}
We consider a class of nonlinear Klein-Gordon equation
$u_{tt}=u_{xx}-u+f(u)$ and show that generically there exist small breathers with
exponentially small tails.
\end{abstract}

\section{Introduction}\label{In}
The nonlinear Klein-Gordon equation
\begin{equation}\label{eq1}
u_{tt}=u_{xx}-u+f(u)\ , \ x\in\R,
\end{equation}
is an important physical model to study relativistic electrons. Classical examples include Sine-Gordon equation and $\phi^4$-model. It is well known that the Sine-Gordon equation, where $f(u)=u-\sin u$, has a family of breather type solutions, which are periodic in time and localized in space. On the one hand, as shown by Birnir-McKean-Weinstein \cite{BMW94} and Denzler \cite{D97},these breathers are rigid in the sense that they do not persist under small perturbations to the sine-Gordon equation. On the other hand, Kruskal-Segur \cite{HK88} use a formal asymptotic expansion to show that the $\phi^4$ model, which can be viewed as a perturbation to the sine-Gordon equation for small amplitude waves, admits  breathers with exponentially small tails. In this manuscript, we carry out rigorous analysis to obtain results in \cite{HK88}. In other words, we show that \eqref{eq1} has a family of small amplitude periodic solutions, which have exponentially small tails with respect to the amplitude of the solution. Our main result is the following.
\begin{MT}
Assume $f$ is odd and holomorphic such that $f(0)=0,f'''(0)>0$.
There exist $\ep_0>0$ and $c>0$ such that for any $\ep\in(0,\ep_0)$,
\eqref{eq1} has a family of solutions $u(x,t)=O(\ep)$ which are $\frac{2\pi}{\sqrt{1-\ep^2}}$-periodic and odd in $t$  such that as $|x|\to\infty$,
\[\Big|u(x,t)\Big|_{H_t^1}\leq Ce^{-\frac c\ep}.\]
\end{MT}

It is advantageous to rewrite \eqref{eq1} by using the spatial dynamics method as an infinite dimensional dynamical system 
\begin{equation}\label{eq1.1}
\left\{\begin{aligned}
&\begin{pmatrix}
u\\v
\end{pmatrix}_t=\begin{pmatrix}
0 & 1\\
1+\p_{xx} & 0
\end{pmatrix}\begin{pmatrix}
u \\ v
\end{pmatrix}-\begin{pmatrix}
0 \\ f(u)
\end{pmatrix}\\
&u(\cdot,t),v(\cdot,t)\ \text{are periodic}.
\end{aligned}\right.
\end{equation}
In this formulation, solutions in the Main Theorem can be viewed as a class of homoclinic orbits which converge to exponentially small solutions in some phase space. The idea of our proof is a combination of partial norm forms and invariant manifold theory. We split the proof into two parts. In the first part, we rescale the spatial, temporal variables and the unknown $u$ in \eqref{eq1.1}, namely,
$x\mapsto\frac {x}{\sqrt{1-\ep^2}}\ ,\ t\mapsto \ep(1-\ep^2)^{\frac 12}t\ ,\ u\mapsto \ep u$, and decompose the unknown $\ep u$ into two parts (hyperbolic and elliptic) according to eigenvalues of the linear operator $\begin{pmatrix}
0 & 1\\
1+(1-\ep^2)\p_{xx} & 0
\end{pmatrix}$. Because of the Hamiltonian nature of our problem, we then perform a sequence of symplectic partial normal form transformations to obtain a system whose oscillatory component is almost invariant up to an error of $O(e^{-\frac c\ep})$. Such procedure yields a Hamiltonian system with two temporal scales of normally elliptic type, i.e., we have a normally elliptic singularly perturbed system. In the second part of the proof, 
we construct local invariant manifolds, namely, center-stable, center-unstable and ect., of sizes $O(1)$ based on our previous work \cite{LZ10}. Since we look for solutions with exponentially small tails, we take $O(e^{-\frac c\ep})$ pieces in center directions of center-stable and center-unstable manifolds and show they intersect by adopting an energy type argument as in \cite{SZ03}, from which we can conclude our Main Theorem.

After this work was completed, the author was informed about \cite{GS01} by Groves-Schneider. The authors show that \eqref{eq1} has small amplitude modulating pulse solutions which can approximated by a pulse-like solution beyond all orders in $\ep$. It is important to note the problem there corresponds to the spatial wave number (in the linear dispersion relation) $k_0\ne0$, which cannot cover our case for $k_0=0$. In the proof, Groves-Schneider use Hamiltonian flows to generate symplectic transformations and construct invariant manifolds with sizes small in $\ep$. Due to some reversibility properties of the system, they are able to prove the intersections of center-stable manifold and center-unstable manifold, which gives solutions they look for. In this paper, we use generating functions introduced in \cite{GL02} to obtain partial normal form transformations and construct $O(1)$ invariant manifolds independent of $\ep$. As we mentioned in last paragraph, we find intersections of center-stable and center-unstable manifolds by an energy type argument in \cite{SZ03}, which needs to be modified accordingly to fit our singularly perturbed system. In \cite{GS05,GS08}, they study similar problems for quasilinear wave equations and obtain $O(\ep^n)$ and $O(e^{-\frac{c}{\sqrt{\ep}}})$ error bounds over spatial length scales of $O(\ep^{-n})$ and $O(e^{\frac{c}{\sqrt{\ep}}})$, respectively. In a more recent preprint \cite{KW12} by Kristiansen-Wulff, the authors study a general dynamical system in some Banach space or a finite dimensional Hamiltonian system with two temporal scales of normally elliptic type. In there, they obtain an $O(e^{-\frac c\ep})$ approximation of the slow manifold. Even though our generating functions are similar to those in \cite{KW12}, we will present all necessary details because of our infinite dimensional Hamiltonian set up. 

The main feature of the problem studied here can thought as a limit system under normally elliptic singular perturbation. Some other existing literature on this topic can be found in \cite{GL02,Ne84,L00,LZ10,Ma01,SZ02} and references therein.

The rest of this paper is organized as follows. In
Section~\ref{P}, we introduce some notations and preliminary
results. In
Section~\ref{Tr}, we define and estimate a sequence of symplectic transformations which transform \eqref{eq1.1} into a
system whose oscillatory components are almost invariant up to an error of $O(e^{-\frac c\ep})$. We construct invariant manifolds for
the transformed system in Section~\ref{IM}. In
Section~\ref{Intersection}, we prove the intersection of invariant
manifolds, which will imply our main theorem. 

\section{Preliminaries}\label{P}
In this section, we introduce some notations and transform \eqref{eq1.1} into a normally elliptic singularly perturbed system which is needed for our future analysis. For $\om>0$, let
\[\begin{aligned}
&L^2(\om)\triangleq\{u=\sum_{n=1}^\infty a_n\sin{(n\om x)}\Big| \sum_{n=1}^{\infty}a_n^2<\infty\},\\
&H^1(\om)\triangleq\{u\in L^2(\om)\Big|\sum_{n=1}^{\infty}(1+n^2\om^2
)a_n^2<\infty\},\end{aligned}\]
i.e., the function spaces are just intersections of $\frac{2\pi}{\om}$-periodic odd functions and sobolev spaces $L^2$ and $H^1$, respectively. For simplicity, when $\om=1$, we write
$L^2=L^2(1)$ and $H^1=H^1(1)$. We assume
\begin{enumerate}
\item[(H)] $f$ is odd and holomorphic in $u$. Moreover,
$f'(0)=0,f'''(0)>0$.
\end{enumerate}
By rescaling $x$ to $\frac x\om$ in \eqref{eq1.1}, we have
\begin{equation}\label{eq2.2}
\left\{\begin{aligned} &u_{t}=v,\\
&v_t=(1+\om^2\p_{xx})u-f(u).
\end{aligned}\right.
\end{equation}
We will study \eqref{eq2.2} as a perturbation problem with small parameter
$\ep=\sqrt{1-\om^2}$. It is standard to show that the linear operator
\[\begin{pmatrix} 0
& 1\\1+\om^2\p_{xx} & 0
\end{pmatrix}\]
has eigenvalues
$\lambda_k^{\pm}=\pm\sqrt{1-k^2\om^2}$ for $k\geq1$. Therefore,
$\lambda_{\pm 1}=\pm\ep$ are real and all others are purely imaginary. Motivated by this fact, we decompose $u$
into hyperbolic part and elliptic part, namely,
\[u(t,x)=u^h(t)\sin{x}+u^c(t,x)\ , \
\int_{-\pi}^{\pi} u^c(t,x)\sin{x}\ dx=0.\]Moreover, we let
\[\tau=\ep(1-\ep^2)^{\frac 12}t\ ,\
 w^h(\tau)=\frac{u^h(\ep^{-1}(1-\ep^2)^{-\frac 12}\tau)}{\ep}\ ,\
w^c(\tau,x)=\frac{u^c(\ep^{-1}(1-\ep^2)^{-\frac 12}\tau,x)}{\ep}.\]
It is easy to verify $w^h$ and $w^c$ satisfy
\begin{equation}\label{eq2.5}
\left\{\begin{aligned}
&(1-\ep^2)w_{\tau\tau}^h-w^h+\frac{1}{\ep^3}P_h f(\ep w^h\sin{x}+\ep
 w^c)=0,\\
&(1-\ep^2)w_{\tau\tau}^c-\frac{1}{\ep^2}\big(1+(1-\ep^2)\p_{xx}\big)w^c+\frac
{1}{\ep^3}P_cf(\ep w^h\sin{x}+\ep  w^c)=0,
\end{aligned}\right.
\end{equation}
where
\[\begin{aligned}
& P_h f(\ep w^h\sin{x}+\ep \tilde w^c)\triangleq \big(\frac
{1}{\pi}\int_{-\pi}^{\pi} f(\ep
w^h\sin{x}+\ep w^c)\sin{x}\ dx\big),\\
& P_c f(\ep w^h\sin{x}+\ep w^c)\triangleq  f(\ep w^h\sin{x}+\ep
 w^c)-\big(P_h f(\ep w^h\sin{x}+\ep w^c)\big)\sin{x}.
\end{aligned}\]
Finally, we set $w_1^h=(1-\ep^2)^{\frac 12}w_\tau^h$,
$w_1^c=\ep(-1-\p_{xx})^{-\frac 12}w_\tau^c$ and plug into
\eqref{eq2.5} to obtain a first order system
\begin{equation}\label{eq2.5.1}
\left\{\begin{aligned} &\begin{pmatrix} w^h\\w_1^h
\end{pmatrix}_{\tau}=\frac {1}{(1-\ep^2)^{\frac 12}}\begin{pmatrix}
0 & 1\\
1 & 0
\end{pmatrix}
\begin{pmatrix}
w^h\\w_1^h
\end{pmatrix}+\begin{pmatrix}
0\\ -\frac{1}{(1-\ep^2)^{\frac 12}\ep^{3}}P_h f(\ep w^h\sin{x}+\ep
w^c)
\end{pmatrix},\\
&\begin{pmatrix} w^c\\ w_1^c
\end{pmatrix}_{\tau}=\frac{1}{\ep}\begin{pmatrix}
0 & (-1-\p_{xx})^{\frac 12}\\
-(-1-\p_{xx})^{\frac 12} & 0
\end{pmatrix}\begin{pmatrix}
w^c\\ w_1^c
\end{pmatrix}\\
&\hspace{1.5cm}+\frac{\ep(-1-\p_{xx})^{-\frac 12}}{1-\ep^2}\begin{pmatrix} 0 \\
w^c-\frac{1}{\ep^3}P_c f(\ep w^h\sin{x}+\ep w^c)
\end{pmatrix},
\end{aligned}\right.
\end{equation}
which is in the singular perturbation form of normally elliptic type
due to fast oscillatory feature of the second equation. The above system is also a Hamiltonian system with Hamiltonian
\begin{equation}\label{eq2.5.1a}
\begin{aligned}&H(w^h,w_1^h,w^c,w_1^c,\ep)\\
=&\pi\Big[\frac {(w_1^h)^2}{2}-\frac{(w_h)^2}{2}+\frac{f^{\prime\prime\prime}(0)}{32}(w^h)^4+\frac{1}{\ep^2\pi}\int_{-\pi}^\pi \frac{\om^2}{2}((-1-\p_{xx})^{\frac 12}w_1^c)^2\\
&\hspace{1cm}+\frac{\om^2}{2}(w_x^c)^2+\frac{F(\ep w^h\sin x+\ep w^c)}{\ep^2}-\ep^2\frac{f^{\prime\prime\prime}(0)}{24}(w^h\sin x)^4-\frac{(w^c)^2}{2}\ dx\Big],
\end{aligned}
\end{equation}
where $F$ is the anti-derivative of $f$ with $F(0)=0$,
so that \eqref{eq2.5.1} can be written as
\begin{equation}\label{eq2.5.3}
\begin{pmatrix}
w^h\\ w_1^h \\ w^c\\ w_1^c
\end{pmatrix}_\tau=\begin{pmatrix}
0 & \frac{1}{\omega\pi} & 0 & 0\\
-\frac{1}{\omega\pi} & 0 & 0 & 0\\
0 & 0 & 0 & \frac{\ep}{\om^2}(-1-\p_{xx})^{-\frac 12}\\
0 & 0 & -\frac{\ep}{\om^2}(-1-\p_{xx})^{-\frac 12} & 0\end{pmatrix}\nabla H.
\end{equation}
To simplify our notation, we write
\begin{equation}\label{eq2.6}
\left.
\begin{aligned}
& X\triangleq \mathbb{R}^2\ , \ Y\triangleq \big(\oplus_{k=2}^{\infty}\{\mathbb{R}^2\sin{kx}\}\big)\cap \big(L^2\times L^2\big),\\
& W^h=\begin{pmatrix} w^h\\w_1^h
\end{pmatrix}\in X\ , \ W^c=\begin{pmatrix}
w^c\\w_1^c
\end{pmatrix}\in Y,\\
&A\triangleq\begin{pmatrix}
0 & 1\\
1 & 0
\end{pmatrix}\ , \ L\triangleq (-1-\p_{xx})^{\frac 12}\ , \ J\triangleq\begin{pmatrix}
0 & L\\
-L& 0
\end{pmatrix},\\
& Y_1\triangleq D(J),\ \ \mbox{the domain of $J$ which is endowed with graph norm},\\
& F(W^h,W^c,\ep)\triangleq(\frac{1}{(1-\ep^2)^{\frac
12}}-1)\begin{pmatrix}
0 & 1\\
1 & 0
\end{pmatrix}W^h\\
&\hspace{2.7cm}+\begin{pmatrix} 0\\ -\frac{1}{(1-\ep^2)^{\frac
12}\ep^{3}}P_h f(\ep w^h\sin{x}+\ep w^c)
\end{pmatrix},\\
&G(W^h,W^c,\ep)\triangleq\frac{(-1-\p_{xx})^{-\frac 12}}{1-\ep^2}\begin{pmatrix} 0 \\
w^c-\frac{1}{\ep^3}P_c f(\ep w^h\sin{x}+\ep w^c)
\end{pmatrix},
\end{aligned}\right.
\end{equation}
where the norm on $Y$ is the standard $L^2$ norm. It is straight
forward to verify that
\begin{equation}\label{eq2.5.2}
J^{\star}=-J\ , \ J^{-1}\in L(Y,Y_1)\ ,\ |J^{-1}|_{L(Y,Y_1)}\leq2,
\end{equation}
where $J^{\star}$ is the adjoint of $J$. 

Using above notations, we
can write \eqref{eq2.5.1} abstractly as
\begin{equation}\label{eq2.6.1}
\left\{
\begin{aligned}
& W_\tau^h=AW^h+F(W^h,W^c,\ep),\\
& W_\tau^c=\frac J\ep W^c+\ep G(W^h,W^c,\ep).
\end{aligned}\right.
\end{equation}

\begin{prop}\label{prop2.1}
Assume (H). For any $0\leq m\in\mathbb{N}$ and $K>0$, there exists
$\ep_0>0$ such that
\begin{equation}\label{eq2.7}
\begin{aligned}
&(D^iF,D^iG)\in C^0(B_K(0,X\times Y_1)\times(-\ep_0,\ep_0),\\
&\hspace{6cm}L(\otimes^i(X\times
Y_1),X\times Y_1)\ ,\ 0\leq i\leq m,\\
&(D^iF,D^iG)\in C^0(B_K(0,X\times Y_1)\times(-\ep_0,\ep_0),\\
&\hspace{4cm}L(\otimes^{i-1}(X\times Y_1)\otimes(X\times Y),X\times Y)\ ,\ 1\leq i\leq m,\\
&(\partial_{\ep}^i F,\partial _\ep^i G)\in C^0(B_K(0,X\times Y_1)\times(-\ep_0,\ep_0),X\times Y_1).\\
\end{aligned}
\end{equation}
Moreover,
\begin{equation}\label{eq2.7.0}
(F,G)(0,0,\ep)=0\ , \ DF(0,0,0)=0\ , \ (D^2F,D^2G)(0,0,\ep)=0.
\end{equation}
Finally, the Cauchy problem of \eqref{eq2.5.1} has a unique mild solution in $X\times
Y_1$.
\end{prop}
\begin{proof}
The smoothness of $F$ and $G$ is a direct consequence of the regularity of $W^h$ and $W^c$ and the analyticity of $f$. The verification of \eqref{eq2.7.0} is straightforward. Finally, the
well-posedness of \eqref{eq2.6} is given by the standard semigroup
theory and Duhamel's principle, see \cite{Pa83}.
\end{proof}

The formal singular limit of \eqref{eq2.5.1}, which can rigorously justified by Proposition \ref{prop2.1} and Theorem 2.2 in
\cite{LZ10}, is given by the
following Duffing equation
\begin{equation}\label{eq2.9}
\begin{pmatrix}
w_0^h\\w_{10}^h
\end{pmatrix}_\tau=\begin{pmatrix}
0 & 1\\
1& 0
\end{pmatrix}\begin{pmatrix}
w_0^h\\w_{10}^h
\end{pmatrix}-\begin{pmatrix}
0 \\ \frac {1}{8} f'''(0)(w_0^h)^3
\end{pmatrix},
\end{equation}
which has a homoclinic orbit
\begin{equation}\label{eq2.10}
h(\tau)=\big(\frac{4}{\sqrt{
f'''(0)}\cosh{\tau}},-\frac{4\tanh{\tau}}{\sqrt{
f'''(0)}\cosh{\tau}}\big).
\end{equation}

\section{Partial Normal Form Transformations}\label{Tr}
The plan of this section is to construct partial normal form transformations for \eqref{eq2.6.1}. We split the whole section into two parts. In the first part, we formally construct a sequence of symplectic transformations close to identity, namely,
\begin{equation}\label{eq3.1}
\begin{aligned}
Z_r=(Z_r^h,Z_r^c)&=(z_r^h,z_{1,r}^h,z_r^c,z_{1,r}^c)\\
&=(I+\Gamma_r(\cdot,\cdot,\ep))\circ\cdots\circ(I+\Gamma_2(\cdot,\cdot,\ep))(W^h,W^c),
\end{aligned}
\end{equation}
so that the equation in the normal direction can be written as
\begin{equation}\label{eq3.2}
\p_\tau Z_r^c=\frac J\ep Z_r^c+\ep G_r(Z_r,\ep)Z_r^c+e^{-r}\tilde
G_r(Z_r,\ep).
\end{equation}
Thus, the center space $\{Z_r^c=0\}$ is almost
invariant up to an error of $O(e^{-r})$. In the second part, we show that $r$ can be taken as $O(\frac 1\ep)$ and provide necessary estimates on those partial normal form transformations and terms appearing in the transformed system.\\
\\
\noindent{{\bf Partial Normal Forms.}} In view of \eqref{eq2.5.1}, one can see $W^h$ corresponds to eigenvalues of $O(1)$ and $W^c$ corresponds to eigenvalues of $O(\frac 1\ep)$. Thus, we expect terms contained in $G(W^h,0,\ep)$ can be removed by some partial norm forms. Let $(G^c,G_1^c)=G$ and recall that $L=(-1-\p_{xx})^{-\frac 12}$. We define the following generating function
\[\begin{aligned}
T_2(w^h,z_{1,2}^h,w^c,z_{1,2}^c)=&<w^c,z_{1,2}^c>-<\ep^2 L^{-1}G_1^c(w^h,z_{1,2}^h,0,\ep),z_{1,2}^c>\\
&+\frac 1\ep <w^h,z_{1,2}^h>-<\ep^2 L^{-1}G^c(w^h,z_{1,2}^h,0,\ep),w^c>,
\end{aligned}\]
where $<\cdot,\cdot>$ represents inner product on $\R^2$ or $H^1$. It follows
\begin{equation}\label{eq3.2.1}
\left\{
\begin{aligned}
z_2^h=&w^h-\ep^3\big(<D_2L^{-1}G_1^c(w^h,z_{1,2}^h,0,\ep),z_{1,2}^c>\\
&\hspace{4.5cm}+< D_2L^{-1}G^c(w^h,z_{1,2}^h,0,\ep),w^c>\big),\\
z_{1,2}^h=&w_1^h+\ep^3\big(<D_1L^{-1}G_1^c(w^h,z_{1,2}^h,0,\ep),z_{1,2}^c>\\
&\hspace{4.5cm}+< D_1L^{-1}G^c(w^h,z_{1,2}^h,0,\ep),w^c>\big),\\
z_2^c=&w^c-\ep^2 L^{-1}G_1^c(w^h,z_{1,2}^h,0,\ep)\ , \ 
z_{1,2}^c=w_1^c+\ep^2 L^{-1} G^c(w^h,z_{1,2}^h,0,\ep).
\end{aligned}\right.
\end{equation}
By the implicit function theorem, there exist positive constants $\delta_2,\ep_2$ and the map $(\Gamma_2^h,\Gamma_2^c)=(\gamma_2^h,\gamma_{1,2}^h,\gamma_2^c,\gamma_{1,2}^c):B_{\delta_2}(0,X\times Y_1)\times[0,\ep_2]\to X\times Y_1$ such that
\begin{equation}\label{eq3.2.2}
\begin{pmatrix}
z_2^h\\z_{1,2}^h\\z_2^c\\z_{1,2}^c
\end{pmatrix}=\begin{pmatrix}
w^h\\ w_1^h\\ w^c\\w_1^c\end{pmatrix}+\ep^2\begin{pmatrix}
\ep\gamma_2^h \\ \ep\gamma_{1,2}^h \\ \gamma_2^c \\ \gamma_{1,2}^c
\end{pmatrix}(W^h,W^c,\ep).
\end{equation}
One can also invert the above transformation to obtain
\begin{equation}\label{eq3.2.3}
\begin{pmatrix}
w^h\\ w_1^h\\ w^c\\w_1^c\end{pmatrix}=\begin{pmatrix}
z_2^h\\z_{1,2}^h\\z_2^c\\z_{1,2}^c
\end{pmatrix}+\ep^2\begin{pmatrix}
\ep\bar\gamma_2^h \\ \ep\bar\gamma_{1,2}^h \\ \bar\gamma_2^c \\ \bar\gamma_{1,2}^c
\end{pmatrix}(Z_2^h,Z_2^c,\ep).
\end{equation}
Plugging \eqref{eq3.2.1}--\eqref{eq3.2.3} into \eqref{eq2.6.1}, we have
\[\begin{aligned}
\p_\tau Z_2^h=&W_\tau^h+\ep^3\Big(\Gamma_2^h(W^h,W^c,\ep)\Big)_\tau\\
=&(1+\ep^3 D_{W^h}\Gamma_2^h)(AW^h+F(W^h,W^c,\ep))+\ep^3D_{W^c}\Gamma_2^h(\frac J\ep W^c+\ep G(W^h,W^c,\ep))\\
=&AZ_2^h+F(Z_2^h,Z_2^c,\ep)+\ep^2\tilde F(Z_2^h,Z_2^c,\ep),\end{aligned}\]
and
\[\begin{aligned}
\p_\tau Z_2^c=&\Big(W^c+\ep^2\begin{pmatrix}
-L^{-1}G_1^c(w^h,z_{1,2}^h,0,\ep)\\ L^{-1}G^c(w^h,z_{1,2}^h,0,\ep)
\end{pmatrix}\Big)_\tau\\
=&\frac J\ep Z_2^c+\ep \Big(G(w^h,w_1^h,W^c,\ep)-G(w^h,z_{1,2}^h,0,\ep)\Big)+\ep^2\begin{pmatrix}
-L^{-1}\p_\tau G_1^c(w^h,z_{1,2}^h,0,\ep)\\ L^{-1}\p_\tau G^c(w^h,z_{1,2}^h,0,\ep)\end{pmatrix}\\
=&\frac J\ep Z_2^c+\ep\big(\int_0^1 D_{W^c}G(z_2^h+\ep^3\bar\gamma_2^h,z_{1,2}^h+p\ep^3\bar\gamma_{1,2}^h,
p(Z_2^c+\ep^2(\bar\gamma_2^c,\bar\gamma_{1,2}^c)),\ep)dp\big)Z_2^c\\
&+(\int_0^1 D_{w_1^h}G\ dp)\ep^4\bar\gamma_{1,2}^h+(\int_0^1 D_{W^c}G\ dp)\ep^3(\bar\gamma_2^c,\bar\gamma_{1,2}^c)+\ep^2\begin{pmatrix}
-L^{-1}\p_\tau G_1^c(w^h,z_{1,2}^h,0,\ep)\\ L^{-1}\p_\tau G^c(w^h,z_{1,2}^h,0,\ep)\end{pmatrix}\\
\triangleq&\frac J\ep Z_2^c+\ep G_2(Z_2,\ep)Z_2^c+\ep^2\tilde G_2(Z_2,\ep).
\end{aligned}\]
In summary, we have constructed a symplectic transformation given by $Z_2^h=W^h+\ep^3\Gamma_2^h(W^h,W^c,\ep),Z_2^c=W^c+\ep^2\Gamma_2^c(W^h,W^c,\ep)$ such that \eqref{eq2.6} is transform into
\begin{equation}\label{eq3.2.4}
\left\{
\begin{aligned}
& \p_\tau Z_2^h=AZ_2^h+F_2(Z_2^h,Z_2^c,\ep),\\
& \p_\tau Z_2^c=\frac J\ep Z_2^c+\ep G_2(Z_2,\ep)Z_2+\ep^2\tilde G(Z_2,\ep),
\end{aligned}\right.
\end{equation}
where
\begin{equation}\label{eq3.2.4a}
\begin{aligned}
&F_2(Z_2^h,Z_2^c,\ep)=F(Z_2^h,Z_2^c,\ep)+\ep^2\tilde F(Z_2^h,Z_2^c,\ep),\\
&G_2(Z_2,\ep)=\int_0^1 D_{W^c}G(z_2^h+\ep^3\bar\gamma_2^h,z_{1,2}^h+p\ep^3\bar\gamma_{1,2}^h,
p(Z_2^c+\ep^2(\bar\gamma_2^c,\bar\gamma_{1,2}^c)),\ep)dp,\\
&\tilde G_2(Z_2,\ep)=(\int_0^1 D_{w_1^h}G\ dp)\ep^4\bar\gamma_{1,2}^h+(\int_0^1 D_{W^c}G\ dp)\ep^3(\bar\gamma_2^c,\bar\gamma_{1,2}^c)\\
&\hspace{5.5cm}+\ep^2\begin{pmatrix}
-L^{-1}\p_\tau G_1^c(z_2^h+\ep^3\gamma_2^h,z_{1,2}^h,0,\ep)\\ L^{-1}\p_\tau G^c(z_2^h+\ep^3\gamma_2^h,z_{1,2}^h,0,\ep)\end{pmatrix}.
\end{aligned}
\end{equation}

After $k$ steps, the system is in the form
\begin{equation}\label{eq3.2.5}
\left\{
\begin{aligned}
& \p_\tau Z_k^h=AZ_k^h+F_k(Z_k,\ep),\\
& \p_\tau Z_k^c=\frac J\ep Z_k^c+\ep G_k(Z_k,\ep)Z_k^c+\ep^k\tilde G_k(Z_k,\ep).
\end{aligned}\right.
\end{equation}
In the next step, we define the generating function
\[\begin{aligned}
T_{k+1}(z_k^h,z_{1,k+1}^h,z_k^c,z_{1,k+1}^c)=&<z_k^c,z_{1,k+1}^c>-<\ep^{k+1} L^{-1}\tilde G_{1,k}^c(z_k^h,z_{1,k+1}^h,0,\ep),z_{1,k+1}^c>\\
&+\frac 1\ep <z_k^h,z_{1,k+1}^h>-<\ep^{k+1} L^{-1}\tilde G_k^c(z_k^h,z_{1,K+1}^h,0,\ep),z_k^c>,
\end{aligned}\]
where $(\tilde G_k^c,\tilde G_{1,k}^c)=\tilde G_k$. It implies
\begin{equation}\label{eq3.2.6}
\left\{
\begin{aligned}
z_{k+1}^h=&z_k^h-\ep^{k+2}\big(<D_2L^{-1}\tilde G_{1,k}^c(z_k^h,z_{1,k+1}^h,0,\ep),z_{1,k+1}^c>\\
&\hspace{4.5cm}+< D_2L^{-1}\tilde G_k^c(z_k^h,z_{1,k+1}^h,0,\ep),z_k^c>\big),\\
z_{1,k+1}^h=&z_{1,k}^h+\ep^{k+2}\big(<D_1L^{-1}\tilde G_{1,k}^c(z_k^h,z_{1,k+1}^h,0,\ep),z_{1,k+1}^c>\\
&\hspace{4.5cm}+< D_1L^{-1}\tilde G_k^c(z_k^h,z_{1,k+1}^h,0,\ep),z_k^c>\big),\\
z_{k+1}^c=&z_k^c-\ep^{k+1} L^{-1}\tilde G_{1,k}^c(z_k^h,z_{1,k+1}^h,0,\ep),\\
z_{1,k+1}^c=&z_{1,k}^c+\ep^{k+1} L^{-1} \tilde G_k^c(z_k^h,z_{1,k+1}^h,0,\ep).
\end{aligned}\right.
\end{equation}
There exist positive constants $\delta_{k+1},\ep_{k+1}$ and the map \[(\Gamma_{k+1}^h,\Gamma_{k+1}^c)=(\gamma_{k+1}^h,\gamma_{1,k+1}^h,\gamma_{k+1}^c,\gamma_{1,k+1}^c):B_{\delta_{k+1}}(0,X\times Y_1)\times[0,\ep_{k+1}]\to X\times Y_1\] such that
\begin{equation}\label{eq3.2.7}
\begin{pmatrix}
z_{k+1}^h\\z_{1,k+1}^h\\z_{k+1}^c\\z_{1,k+1}^c
\end{pmatrix}=\begin{pmatrix}
z_k^h\\ z_{1,k}^h\\ z_k^c\\z_{1,k}^c\end{pmatrix}+\ep^{k+1}\begin{pmatrix}
\ep\gamma_{k+1}^h \\ \ep\gamma_{1,k+1}^h \\ \gamma_{k+1}^c \\ \gamma_{1,k+1}^c
\end{pmatrix}(Z_k^h,Z_k^c,\ep).
\end{equation}
By inverting the above transformation, we obtain
\begin{equation}\label{eq3.2.8}
\begin{pmatrix}
z_k^h\\ z_{1,k}^h\\ z_k^c\\z_{1,k}^c\end{pmatrix}=\begin{pmatrix}
z_{k+1}^h\\z_{1,k+1}^h\\z_{k+1}^c\\z_{1,k+1}^c
\end{pmatrix}+\ep^{k+1}\begin{pmatrix}
\ep\bar\gamma_{k+1}^h \\ \ep\bar\gamma_{1,k+1}^h \\ \bar\gamma_{k+1}^c \\ \bar\gamma_{1,k+1}^c
\end{pmatrix}(Z_{k+1}^h,Z_{k+1}^c,\ep).
\end{equation}
Finally, we plug \eqref{eq3.2.6}--\eqref{eq3.2.8} into \eqref{eq3.2.5} to obtain
\[\begin{aligned}
\p_\tau Z_{k+1}^h=&\p_\tau Z_k^h+\ep^{k+2}\Big(\Gamma_{k+1}^h(Z_k,\ep)\Big)_\tau\\
=&(1+\ep^{k+2} D_{Z_k^h}\Gamma_2^h)(AZ_k^h+F_k(Z_k,\ep))+\ep^{k+2}D_{Z_k^c}\Gamma_{k+1}^h(\frac J\ep Z_k^c+\ep G_k(Z_k,\ep)+\ep^k\tilde G_k(Z_k,\ep))\\
=&AZ_k^h+F_k(Z_{k+1},\ep)+\ep^{k+1}\tilde F(Z_{k+1},\ep)\triangleq AZ_{k+1}^h+F_{k+1}(Z_{k+1},\ep),\end{aligned}\]
and
\begin{equation}\label{eq3.2.9}
\begin{aligned}
&\p_\tau Z_{k+1}^c\\=&\Big(Z_k^c+\ep^{k+1}\begin{pmatrix}
-L^{-1}\tilde G_{1,k}^c(z_k^h,z_{1,k+1}^h,0,\ep)\\ L^{-1}\tilde G_k^c(z_k^h,z_{1,k+1}^h,0,\ep)
\end{pmatrix}\Big)_\tau\\
=&\frac J\ep Z_k^c+\ep G_k(Z_k,\ep)Z^c_k+\ep^k\tilde G_k(Z_k,\ep) +\ep^{k+1}\begin{pmatrix}
-L^{-1}\p_\tau\tilde G_{1,k}^c(z_k^h,z_{1,k+1}^h,0,\ep)\\ L^{-1}\p_\tau\tilde G_k^c(z_k^h,z_{1,k+1}^h,0,\ep)
\end{pmatrix}\\
=&\frac J\ep Z_{k+1}^c+\ep G_k(Z_{k+1}+O(\ep^{k+1}),\ep)(Z_{k+1}^c+O(\ep^{k+1}))\\
&+\ep^k\Big[\tilde G_k(z_{k+1}^h+O(\ep^{k+2}),z_{1,k+1}^h+O(\ep^{k+2}),Z_{k+1}^c+O(\ep^{k+1}),\ep)\\
&-\tilde G_k(z_{k+1}^h+O(\ep^{k+2}),z_{1,k+1}^h,0,\ep)\Big]+\ep^{k+1}\begin{pmatrix}
-L^{-1}\p_\tau\tilde G_{1,k}^c(z_k^h,z_{1,k+1}^h,0,\ep)\\ L^{-1}\p_\tau\tilde G_k^c(z_k^h,z_{1,k+1}^h,0,\ep)
\end{pmatrix}\\
=&\frac J\ep Z_{k+1}^c+\Big(\ep G_k(Z_{k+1},\ep)+\ep^k\int_0^1 D_{Z_k^c}\tilde G_k(Z_{k+1},p,\ep)\ dp\Big)Z_{k+1}^c+O(\ep^{k+1})\\
\triangleq&\frac J\ep Z_{k+1}^c+\ep G_{k+1}(Z_{k+1},\ep)Z_{k+1}^c+\ep^{k+1}\tilde G_{k+1}(Z_{k+1},\ep).
\end{aligned}\end{equation}

In the second part of this section, we use Cauchy integrals to control bounds of $(\Gamma_i^h,\Gamma_i^c)$ as well as terms $F_i,G_i$ and $\tilde G_i$ appearing in the transformed systems. We will go through this procedure inductively by shrinking the spatial domain at each step. As a technical point, since each symplectic transformation depends on $\ep$, it is necessary to keep a uniform domain for $\ep$, i.e., there exists $\ep_0>0$ such that the above functions are well defined on $[0,\ep_0]$ for all $i$. This can be achieved by having a uniform bound on $\tilde G_i$, which will be proved in the following.\\

\noindent{{\bf Estimates on Partial Normal Forms and the Transformed System \eqref{eq3.2.5}}.}The estimates are obtained by taking advantage of the
analyticity of $f$. First, we complexify the domain of $F,G$. Let
\[X^{\mathbb{C}}\triangleq X\oplus iX\ , \
Y_1^{\mathbb{C}}\triangleq Y_1\oplus iY_1,\]where the norm is given
by the sum of the real and imaginary part. With slight abuse of
notation, we still use $X$ and $Y_1$ to denote $X^{\mathbb{C}}$ and
$Y_1^{\mathbb{C}}$. We introduce the following notations for our
next lemma. Let
\begin{equation}\label{eq3.6}
P(s)=G(W^h+s\delta W^h,W^c,\ep)\ , \ Q(s)=G(W^h,W^c+s\delta W^c,\ep),
\end{equation}
where $s\in\mathbb{C}$ and $|\delta W^h|_X=|\delta W^c|_{Y_1}=1$.
Given any $K>0,K_1>0$ and $m\geq2$, we denote
\[
\Omega_m=B_{2K-(m-1)K_1\ep} (0,X\times Y_1),
\]which is the ball in the function space $X\times Y_1$ centered at the origin with radius $2K-(m-1)K_1\ep$. For sufficiently small $c>0$ and $m\leq [\frac c\ep]$, we
have
\begin{equation}\label{eq3.6.1}
2K-(m-1)K_1\ep>K\ ,\ \emptyset\neq\Omega_{m+1}\subset\Omega_m,
\end{equation}
where $[\cdot]$ denotes the largest integer that is less than or
equal to the number in the bracket.
\begin{lemma}\label{le3.1}
Assume (H). Let $K=|h|_{C_\tau^0}+1$, where $h$ is given
in \eqref{eq2.10}. There exists $\ep_0>0,c>0$ such that for each
$\ep\in[0,\ep_0)$ and $2\leq m\leq [\frac c\ep]$,
\begin{equation}\label{eq3.7}
|\ep G_m|_{C^0(\Omega_{m},L(Y_1,Y_1))}\leq 1-\frac{1}{2^m}\ , \
|\ep^m\tilde G_m|_{C^2(\Omega_{m})}\leq \frac{\ep}{2^m}.
\end{equation}
\end{lemma}
\begin{proof}
We will prove \eqref{eq3.7} inductively. Let
\[C_1=(|F|_X+|G|_{Y_1})_{C^0(B_{2K}(0,X\times Y_1))}.\] 
When $m=2$, since $G$ is analytic, $P(s)$ is an analytic function of $s$ for $|s|\leq\frac{K_1\ep}{2}$ and $|W^h|_X+|W^c|_{Y_1}\leq 2K-\frac{K_1\ep}{2}$. The Cauchy integral gives
\[\begin{aligned}
\p_sP(0)&=D_{W^h} G(W^h,W^c,\ep)(\delta W^h)\\
&=\Big(\frac{1}{2\pi
i}\oint_{\p
B_{\frac{K_1\ep}{2}}(0,\mathbb{C})}\frac{G(W^h+s\delta W^h,W^c,\ep)}{s^2}\ ds\Big)(\delta W^h),\\
\p_sQ(0)&=D_{W^c} G(W^h,W^c,\ep)(\delta W^c)\\
&=\Big(\frac{1}{2\pi
i}\oint_{\p
B_{\frac{K_1\ep}{2}}(0,\mathbb{C})}\frac{G(W^h,W^c+s\delta W^c,\ep)}{s^2}\ ds\Big)(\delta W^c),\end{aligned}\]
which implies 
\begin{equation}\label{eq3.7.1}
|DG(\cdot,\cdot,\ep)|_{C^0(B_{2K-\frac{K_1\ep}{2}}(0,X\times Y_1)),L(X\times Y_1, Y_1)}\leq\frac{2C_1}{K_1\ep}.
\end{equation}
By choosing $\ep_0$ sufficiently small, for $\ep\in[0,\ep_0]$,
\[2K-\frac{K_1\ep}{2}-\ep^3\frac{2C_1}{K_1\ep}2K>2K-K_1\ep.\]
Consequently, from \eqref{eq3.2.1}, we conclude that there exists a unique analytic function $\Gamma_2$ of $O(\ep^2)$ such that
$I+\Gamma_2(\cdot,\cdot,\ep)$ is a diffeomorphism between $B_{2K-\frac{K_1\ep}{2}}(0,X\times Y_1)$ and a region that contains at least $\Omega_2$. According to \eqref{eq3.7.1} and the definition of $G_2$ and $\tilde G_2$ in \eqref{eq3.2.4a}, for sufficiently large $K_1$, we have
\[\begin{aligned}
&|\ep G_2|_{C^0(\Omega_{2},L(Y_1,Y_1))}\leq\frac{2C_1}{K_1}<1-\frac{1}{4},\\
&|\ep^2\tilde G_2|_{C^0(\Omega_2)}\leq O(\ep^2)+\ep^2|J^{-1}|_{L(Y,Y_1)}\frac{2C_1}{K_1\ep}(2K+C_1+O(\ep))\leq\ep\frac{8C_1(2K+C_1)}{K_1}<\frac\ep4,
\end{aligned}\]
which completes the proof for $m=2$.
\begin{rem}
In the estimation of $\ep^2\tilde G_2$, we use the fact that $DG$ is also bounded by $\frac{2C_1}{K_1\ep}$ when it is considered as a mapping from $X\times Y$ to $ Y$, which can be obtained by the same proof of \eqref{eq3.7.1}.
\end{rem}

Suppose for $m=k$, we have
\begin{equation}\label{eq3.8}
|\ep G_{k}|_{C^0(\Omega_{k},L(Y_1,Y_1))}\leq
1-\frac{1}{2^{k}}\ , \ |\ep^{k}\tilde
G_{k}|_{C^0(\Omega_{k},Y_1)}\leq \frac{\ep}{2^{k}}.
\end{equation}
Again, by using the Cauchy integral, we have
\[
|D\tilde G_{k}|_{C^0(B_{2K-kK_1\ep-\frac{K_1\ep}{2}},L(X\times
Y_1,Y_1))}\leq\frac{2|\tilde G_k|_{C^0(\Omega_k,Y_1)}}{K_1\ep}.
\]
It follows from \eqref{eq3.2.6} that there exists a unique analytic function $\Gamma_{k+1}$ of $O(\frac{\ep^2}{2^k})$ such that
$I+\Gamma_{k+1}(\cdot,\cdot,\ep)$ is a diffeomorphism between $B_{2K-kK_1\ep-\frac{K_1\ep}{2}}(0,X\times Y_1)$ and a region that contains at least $\Omega_{k+1}$. According to the definition of $G_{k+1}$ and $\tilde G_{k+1}$ in \eqref{eq3.2.9}, for sufficiently large $K_1$, we have
\[|\ep G_{k+1}|_{C^0(\Omega_{k+1},L(Y_1,Y_1))}\leq 1-\frac{1}{2^k}+\frac{\ep}{2^k}\frac{2}{K_1\ep}<1-\frac{1}{2^{k+1}},\]
and
\[\begin{aligned}
|\ep^{k+1}\tilde G_{k+1}|_{C^0(\Omega_{k+1})}&\leq \big(|\ep G_k|_{C^0}+2\ep^k|D\tilde G_k|_{C^0}\big)\frac{\ep^2}{2^k}+\ep^{k+1}|D\tilde G_k||J^{-1}|(2K+C_1+O(\ep)),\\
&\leq(1-\frac{1}{2^k}+\frac{2\ep}{2^k})\frac{\ep^2}{2^k}+\frac{\ep}{2^k}\frac{8C_1(2K+C_1)}{K_1}\leq\frac{\ep}{2^{k+1}}.
\end{aligned}\]
The proof is completed.
\end{proof}
\begin{rem}
In fact, one can show $|\ep G_m|\leq\delta(1-\frac{1}{2^m})$ for any small $\delta>0$ by taking $K_1$ sufficiently large.
\end{rem}
From \eqref{eq3.6.1}, one can see we can perform the partial normal form transformations $[\frac c\ep]$ times, where $c$ is some small positive number. In view of \eqref{eq3.1}, we have
\[\Big|(I+\Gamma_{[\frac c\ep]}(\cdot,\cdot,\ep))\circ\cdots\circ(I+\Gamma_2(\cdot,\cdot,\ep))\Big|\leq 1+\sum_{k=2}^{[\frac
c\ep]}\frac{\ep^2C}{2^k}\leq 1+\ep^2C,\]
where $C$ is independent of $c$ and $\ep$. Thus, the transformation in \eqref{eq3.1} is well defined. 

It is clear that $F_{[\frac c\ep]},G_{[\frac c\ep]}$ and $\tilde G_{[\frac c\ep]}$ are still analytic functions in terms of $Z_{[\frac c\ep]}$. By shrinking the domain to $B_K(0,X\times Y_1)$, we conclude that there exists $C$ such that
\[\begin{aligned}
&|F_{[\frac c\ep]}|_{C^2(B_K(0,X\times Y_1),X)}\leq C,\\
&|\ep G_{[\frac c\ep]}|_{C^2(B_K(0,X\times Y_1),L(Y_1,Y_1))}\leq C,\\
& |\ep^{[\frac c\ep]}\tilde
G_{[\frac c\ep]}|_{C^2(B_K(0,X\times Y_1),Y_1)}\leq C\ep e^{-\frac{c\log{\sqrt{2}}}{\ep}}.
\end{aligned}\]

With slight abuse of notations, we still use $W^h,W^c$ to denote
spatial variables in hyperbolic and elliptic directions and write
the transformed system as
\begin{equation}\label{eq3.14}
\left\{
\begin{aligned}
& W_\tau^h=AW^h+\tilde F(W^h,W^c,\ep),\\
& W_\tau^c=\frac J\ep W^c+\bar
G(W^h,W^c,\ep)W^c+e^{\alpha(\ep)}\tilde G(W^h,W^c,\ep),
\end{aligned}\right.
\end{equation}
where $\alpha(\ep)=-\frac{c\log{\sqrt{2}}}{\ep}$ and
\begin{equation}\label{eq3.15}
\ep|\tilde F|_{C^2}+\ep|\bar G|_{C^2}+|\tilde G|_{C^2}\leq C\ep.
\end{equation}
Moreover, it is straightforward to verify
\begin{equation}\label{eq3.16}
\tilde F(0,0,\ep)=0\ , \ D\tilde F(0,0,0)=0\ , \ \bar G(0,0,0)=0\ ,\ \tilde G(0,0,\ep)=0.
\end{equation}

\section{Invariant Manifolds}\label{IM}
In this section, we study invariant manifolds of \eqref{eq3.14}
and their approximations. More precisely, we first consider a regular perturbation problem of \eqref{eq2.9} and show that it can serve as the leading order approximation of \eqref{eq3.14}. Then we construct various local invariant manifolds of the regular perturbation problem and \eqref{eq3.14} and compare them in terms of $\ep$.\\
\\
\noindent{\bf Leading Order Approximation.} We consider a regular
perturbation problem
\begin{equation}\label{eq4.1}
W_{\star\tau}^h=AW_\star^h+\tilde F(W^h,0,\ep)
\end{equation}
and its linearized problem
\begin{equation}\label{eq4.6}
\left\{\begin{aligned} &(\delta W_{\star}^h)_\tau=A \delta
W_\star^h+D_1\tilde F(W_\star^h,0,\ep)\delta W_\star^h,\\
&(\delta W_\star^c)_\tau=\frac J\ep\delta W_\star^c+ \overline
G(W_\star^h,0,\ep)\delta W_\star^c,
\end{aligned}\right.
\end{equation}
where the $\delta W_\star^c$-equation is used to track the evolution
in normal directions. Since $J$ is anti-selfadjoint on $Y_1$ and $\bar G(W^h,0,\ep)\in
L(Y_1,Y_1)$, $\frac J\ep+\bar G(W^h,0,\ep)$ generates an
evolutionary operator $E(t,s;\ep)$ on $Y_1$ with
\[|E(\tau,s;\ep)|_{L(Y_1,Y_1)}\leq e^{C|\tau-s|},\] see a proof in
\cite{Pa83}. In the following, we will show that solutions of \eqref{eq3.14} and its linearization can be approximated by \eqref{eq4.1} and
\eqref{eq4.6}. Let $\tilde G_\star^\ep(W^h)\triangleq\tilde
G(W^h,0,\ep)$, we have

\begin{theorem}\label{thm4.3}
Let $T>0$ and $(W^h(\tau),W^c(\tau))$ and
$W_\star^h(t)$ be solutions of \eqref{eq3.14} and \eqref{eq4.1} with
$W^h(0)=W_\star^h(0)$ and $W^c(0)=0$. Then there exists $C>0$
depending on $T$ such that
\begin{equation}\label{eq4.4}
|W^h(\tau)-W_\star^h(\tau)|_X+|W^c(\tau)|\leq
Ce^{\alpha(\ep)},
\end{equation}
for all $\tau\in[0,T]$.
\end{theorem}
\begin{proof}
Using Duhamel's principle, we have
\[\begin{aligned}
&(W^h-W_\star^h)(\tau)=\int_0^\tau e^{(\tau-s)A}(\tilde
F(W^h,W^c,\ep)-\tilde F(W^h,0,\ep))\ ds,\\
&W^c(\tau)=\int_0^\tau e^{(\tau-s)\frac J\ep}\big(\overline
G(W^h,W^c,\ep)W^c+e^{\alpha(\ep)}\tilde G(W^h,W^c,\ep)\big)\ ds,
\end{aligned}\]
which implies
\[\begin{aligned}
&|(W^h-W_\star^h)(\tau)|_X\leq\int_0^\tau
e^{(\tau-s)}|D\tilde F|_{C^0}(|W^h-W_\star^h|_X+|W^c|_{Y_1})\ ds,\\
&|W^c(\tau)|_{Y_1}\leq\int_0^\tau (|\overline
G|_{C^0}+e^{\alpha(\ep)}|D\tilde
G|_{C^0})(|W^h-W_\star^h|_X+|W^c|_{Y_1})+e^{\alpha(\ep)}|\tilde
G_\star^\ep|_{C^0}\ ds.
\end{aligned}\]
Consequently, the Gronwall's inequality gives
\[|(W^h-W_\star^h)(\tau)|_X+|W^c(\tau)|_{Y_1}\leq C
e^{\alpha(\ep)}.\]
\end{proof}
\begin{rem}\label{rem4.4}
If $|W^h(0)-W_\star^h(0)|_X+|W^c|_{Y_1}\leq Ce^{\alpha(\ep)}$, the
same result still holds.
\end{rem}

Linearizing \eqref{eq3.14}, we obtain
\begin{equation}\label{eq4.5}
\left\{\begin{aligned} \delta W_\tau^h=&A\delta W^h+D_1\tilde
F\delta W^h+D_2\tilde F\delta
W^c,\\
\delta W_\tau^c=&\frac J\ep\delta W^c+\bar G\ \delta W^c+
D_1\bar G(\delta W^h,W^c)+ D_2\bar G(\delta W^c,W^c)\\
&+e^{\alpha(\ep)}(D_1\tilde G\ \delta W^h+D_2\tilde G\ \delta W^c),
\end{aligned}\right.
\end{equation}
where $D_1\tilde F,D_2\tilde F,\bar G,D_1\bar G,
D_2\bar G,D_1\tilde G,D_2\tilde G$ are evaluated at
$(W^h,W^c,\ep)$.

\begin{theorem}\label{thm4.5}
Assume the same conditions in Theorem \ref{thm4.3}. In addition, we
assume
\[\delta W^h(0)=\delta W_\star^h(0)\ , \ \delta W^c(0)=\delta W_\star^c(0)
\ , \ |\delta W_\star^h(0)|_X+|\delta W_\star^c(0)|_{Y_1}\leq 1.\]
For any $T>0$, there exists $C>0$ depending on $T$ such that for $\tau\in[0,T]$,
\begin{equation}\label{eq4.7}
|\delta W^h(\tau)-\delta W_\star^h(\tau)|_X\leq C\ep\ , \ |\delta W^c(\tau)-\delta
W_\star^c(\tau)|_{Y_1}\leq Ce^{\alpha(\ep)}.
\end{equation}
\end{theorem}
\begin{proof}
By standard ODE theory in Banach space, we have
\begin{equation}\label{eq4.8}
|(\delta W^h,\delta W^c)+(\delta W_\star^h,\delta
W_\star^c)|_{C^0([0,T],X\times Y_1)}\leq C.
\end{equation}
First we use \eqref{eq4.5} and \eqref{eq4.6} to obtain
\begin{eqnarray}\label{eq4.8.1}
(\delta W^h-\delta W_\star^h)_\tau&=&A(\delta W^h-\delta
W_\star^h)+D_1\tilde
F(W_\star^h,0,\ep)(\delta W^h-\delta W_\star^h)\\\nonumber
&&+D_2\tilde  F(W^h,0,\ep)\delta W^c+h_1(t,\ep),\\\label{eq4.8.2}
(\delta W^c-\delta W_\star^c)_\tau&=&(\frac J\ep+\overline
G(W_\star^h,0,\ep))(\delta W^c-\delta W_\star^c)+h_2(t,\ep).
\end{eqnarray}
Due to \eqref{eq4.4}, we know $h_{1,2}$ satisfy
\begin{eqnarray*}
|h_1(t,\ep)|_X &\leq& |D^2\tilde F|_{C^0}(|W^h-W_\star^h|_X+|W^c|_{Y_1})(|\delta
W^h|_X+|\delta W^c|_{Y_1})\leq Ce^{\alpha(\ep)},\\
|h_2(t,\ep)|_{Y_1} &\leq& (e^{\alpha(\ep)}|D\tilde
G|_{C^0}+|D\bar G|_{C^0}|W^c|_{Y_1})(|\delta W^h|_X+|\delta
W^c|_{Y_1})\\
&&+|D\bar G|_{C^0}(|W^h-W_\star^h|_X+|W^c|_{Y_1})|\delta W^c|_{Y_1}\leq Ce^{\alpha(\ep)}.
\end{eqnarray*}
Applying Duhamel's principle to \eqref{eq4.8.2} yields
\[|\delta W^c(\tau)-\delta W_\star^c(\tau)|_{Y_1} \leq Ce^{\alpha(\ep)}\ , \ \tau\in[0,T].\]
To finish the proof, it suffices to show the $O(1)$ term $D_2\tilde F(W^h,0,\ep)\delta W^c$ becomes $O(\ep)$ after integration, namely,
\[|\int_0^\tau e^{(\tau-s)A}D_2\tilde F(W^h(s),0,\ep)\delta W^c(s)\ ds|\leq C\ep.\]
We note from the second equation of \eqref{eq4.5} that $\delta W^c=\ep J^{-1}(\delta W^c)_\tau+O(\ep)$. Therefore,
\[\begin{aligned}
&|\int_0^\tau e^{(\tau-s)A}D_2\tilde F(W^h(s),0,\ep)\delta W^c(s)\ ds|\\
\leq& |\int_0^t e^{(\tau-s)A}\p_\tau D_2\tilde F(W^h(s),0,\ep)\ep J^{-1}\delta W^c(s)\ ds|+O(\ep)\leq C\ep.
\end{aligned}\]
\end{proof}
\begin{rem}\label{rem4.6}
In the above theorem, if we instead assume
\[|\delta W^h(0)-\delta W_\star^h(0)|_X\leq C\ep\ , \ |\delta W^c(0)-\delta W_\star^c(0)|_{Y_1}\leq Ce^{\alpha(\ep)},\] the same result still holds. Furthermore, if we assume
\[\delta W_\star^c(0)=0\ , \ |\delta W^h(0)-\delta W_\star^h(0)|_X+|\delta W^c(0)-\delta W_\star^c(0)|_{Y_1}\leq Ce^{\alpha(\ep)},\] 
the same proof implies $|\delta W^h(\tau)-\delta W_\star^h(\tau)|\leq Ce^{\alpha(\ep)}$.
\end{rem}

\noindent{\bf Invariant Manifolds.} From \eqref{eq3.16}, we know the origin is a fixed point of 
\eqref{eq3.14}. We shall use the Lyapunov-Perron integral equation to construct
various local invariant manifolds around the fixed point. First we
write $X=P_sX\oplus P_u X\triangleq X^s\oplus X^u$, where $X_{s,u}$
are eigenspaces corresponding to eigenvalues $\pm1$ of $A$. Clearly,
$Y_1$ should be considered as the center subspace. Since $J$ is
anti-selfadjoint on $Y_1$, $J$ generates a unitary group on $Y_1$. Thus,
\begin{equation}\label{eq4.9}
|e^{\tau A}\big|_{X^s}|\leq e^{-\tau}\ , \tau\geq0\ , \ |e^{\tau
A}\big|_{X^u}|\leq e^{\tau}\ , \tau\leq0\ , \ |e^ {\tau J}|\leq 1\ ,
\tau\in\mathbb{R}.
\end{equation}
Since we are working in a Hilbert space,
there always exist smooth cut-off functions. In the construction of center-stable and
center-unstable manifolds, we follow the standard procedure to modify nonlinear terms outside a
neighborhood of the fixed point so that they have
global small Lipschitz constants. With slight abuse of
notation, we still use the same notation after multiplying cut-off functions. Define transformations
$\mathscr{T}_{cs(cu)}^\ep(\cdot,\cdot,W_{cs(cu)})$ for
$W_{cs(cu)}=(W_{s(u)},W_c)\in X_{s(u)}\times Y_1$ as
\begin{eqnarray*}
&&\mathscr{T}_{cs(cu)}^\ep(W^h,W^c;W_{cs(cu)})(\tau)\\
&\triangleq& \begin{pmatrix} e^{\tau A}W_{s(u)}\\ e^{\tau\frac
J\ep}W_c
\end{pmatrix}+\int_0^\tau\begin{pmatrix}
e^{(\tau-s)A}P_{s(u)}\tilde F(W^h,W^c,\ep)\\
e^{(\tau-s)\frac J\ep}\big(\bar
G(W^h,W^c,\ep)W^c+e^{\alpha(\ep)}\tilde G(W^h,W^c,\ep)\big)
\end{pmatrix}\ ds\\
&&+\int_{+\infty(-\infty)}^\tau\begin{pmatrix}
e^{(\tau-s)A}P_{u(s)}\tilde F(W^h,W^c,\ep)\\
0
\end{pmatrix}\ ds.
\end{eqnarray*}
For $\eta\in\mathbb{R}$, we define function spaces
\begin{equation}\label{eq4.9.1}
\begin{aligned}
B_\eta^{\pm}(\rho)\triangleq \Big\{(W^h,W^c)\in C^0(\R^{\pm},&X\times
Y_1)\big|\\
&\sup_{\tau\geq0(\leq0)} e^{-\eta
\tau}(|W^h(\tau)|_X+|W^c(\tau)|_{Y_1})<\rho\Big\}\end{aligned}
\end{equation}
with norm
\[|(W^h,W^c)|_\eta^{\pm}=\sup_{\tau\geq0(\leq0)}
e^{-\eta \tau}(|W^h(\tau)|_X+|W^c(\tau)|_{Y_1}).\]
We also use $|(\cdot,\cdot)|_\eta^\pm$ to denote the
norm of bounded linear operators from $X^s\times Y_1$ to
$B_\eta^\pm(\infty)$, where $B_\eta^\pm(\infty)$ denotes the
corresponding linear spaces defined in \eqref{eq4.9.1}.

It is straightforward to verify that given any $r>0$ there exists
$\rho(r)$ such that $\mathscr{T}_{cs}^\ep$ defines a contraction
mapping on $B_\eta^+(\rho(r))$ for $\eta$ in any compact subset of
$[0,1]$ and $|W_{cs}|_{X\times Y_1}<r$. Similar result holds for
$\mathscr{T}_{cu}^\ep$ on $B_\eta^-(\rho)$ with $\eta$ in any
compact subset of $[-1,0]$. Let $(W^h,W^c)$ be the fixed point of
$\mathscr{T}_{cs(cu)}^\ep(\cdot,\cdot,W_{cs(cu)},\ep)$ in
$B_{\pm}(\rho)$ and
\[h_{u(s)}(W_{s(u)},W_c,\ep)\triangleq P_{u(s)}\mathscr{T}_{cs(cu)}^\ep(W^h,W^c;W_{cs(cu)})(0).\]
We define
\[
\mathcal{M}_{cs(cu)}^\ep\triangleq\big\{W_{s(u)}+W_c+h_{u(s)}(W_{s(u)},W_c,\ep)\big||W_{s(u)}|_X+|W_c|_{Y_1}<r\big\}.
\]
Therefore, $\mathcal{M}_{cs(cu)}^\ep$ are global center-stable and
center-unstable manifolds of the origin of the modified version of \eqref{eq3.14}. By
choosing $r$ small enough, $\mathcal{M}_{cs(cu)}^\ep$ are local
ceneter-stable and center-unstable manifolds of $(0,0)$ of
\eqref{eq3.14}. One should note that $\mathcal{M}_{cs(cu)}^\ep$ are
well-defined in a $O(1)$ neighborhood, which is crucial to our
analysis when $\ep\to 0$.

In the construction of stable and unstable manifolds, since we are
looking for solutions with truely exponential decay forward and
backward in time, there is no need to modify nonlinear terms. As a
consequence, we obtain the uniqueness of stable and unstable
manifolds. We define transformations
\begin{eqnarray*}
&&\mathscr{T}_{s(u)}^\ep(W^h,W^c;W_{s(u)})(\tau)\\
&\triangleq& \begin{pmatrix} e^{\tau A}W_{s(u)}\\ 0
\end{pmatrix}+\int_0^\tau\begin{pmatrix}
e^{(\tau-s)A}P_{s(u)}\tilde F(W^h,W^c,\ep)\\
0
\end{pmatrix}\ ds\\
&&+\int_{+\infty(-\infty)}^\tau\begin{pmatrix}
e^{(\tau-s)A}P_{u(s)}\tilde F(W^h,W^c,\ep)\\
e^{(\tau-s)\frac J\ep}\big(\overline
G(W^h,W^c,\ep)W^c+e^{\alpha(\ep)}\tilde G(W^h,W^c,\ep)\big)
\end{pmatrix}\ ds,
\end{eqnarray*}
which are contraction on $B_\eta^{\pm}(\rho)$. Let
$(W^h,W^c)$ be the fixed point of $\mathscr{T}_{s(u)}^\ep$ and
\[h_{cu(cs)}(W_{s(u)},\ep)\triangleq(I-P_{s(u)})\mathscr{T}_{s(u)}^\ep(W^h,W^c;W_{s(u)})(0).\]
The stable and unstable manifolds are given by
\[\mathcal{M}_{s(u)}^\ep\triangleq\big\{W_{s(u)}+h_{cu(cs)}(W_{s(u)},\ep)\big||W_{s(u)}|_X<r\big\}.\]

For the regular perturbation problem \eqref{eq4.1}, we consider
\begin{eqnarray*}
\mathscr{T}_{s(u)}^\star(W_\star^h;W_{s(u)})(\tau)&\triangleq&
e^{tA}W_{s(u)} +\int_0^\tau e^{(\tau-s)A}P_{s(u)}\tilde
F(W_\star^h,0,\ep)\ ds\\
&&+\int_{+\infty(-\infty)}^\tau
e^{(\tau-s)A}P_{u(s)}\tilde F(W_\star^h,0,\ep)\ ds.
\end{eqnarray*}
Let $W_\star^h$ be the fixed point of
$\mathscr{T}_{s(u)}^\star(\cdot;W_{s(u)})$ on function spaces
\[\Big\{W^h\in C^0(\R^{\pm},X)\big|
\sup_{\tau\geq0(\leq0)} e^{-\eta
\tau}|W^h(\tau)|_X<\rho\Big\}\] and let
\[h_{u(s)}^\star(W_{s(u)})\triangleq P_{u(s)}\mathscr{T}_{s(u)}^\ep(W_\star^h;W_{s(u)})(0).\]
We define the stable and unstable manifold of \eqref{eq4.1} as
\[
\mathcal{M}_{s(u)}^\star\triangleq\big\{W_{s(u)}+h_{u(s)}^\star(W_{s(u)})\big||W_{s(u)}|_X+|W_c|_{Y_1}<r\big\}.
\]
In the following, we give our main theorem on center-stable and
unstable manifolds of \eqref{eq3.14}. Similar results and estimates also hold for
center-unstable and stable manifolds.
\begin{theorem}\label{thm4.7}
Assume (H). For the system \eqref{eq3.14}, we have
\begin{enumerate}
\item[1)] There exist $r>0$, $\ep_0>0$ and a mapping $h_u:B_r(0,X^s\times
Y_1)\times(0,\ep_0)\longrightarrow X^u$ such that its graph
$\mathcal{M}_{cs}^\ep$ forms a local center-stable manifold of
the origin.

\item[2)] $h_u$ is $C^2$ in $W_s$ and $W_c$ with norms independent of $\ep$. Moreover,
there exists $C$ independent of $\ep$ such that
\begin{equation}\label{eq4.10}
\begin{aligned}
&h_u(0,0,\ep)=0\ , \ |h_u(\cdot,0,\ep)-h_u^\star(\cdot)|_{C^1}\leq
Ce^{\alpha(\ep)}\ , \ |D_2h_u(W_s,0,\ep)|_{C^0}\leq C\ep.
\end{aligned}
\end{equation}

\item[3)] There exist $r>0$, $\ep_0>0$ and a mapping $h_{cs}:B_r(0,X^u)\times(0,\ep_0)\longrightarrow X^s\times Y_1$
such that its graph $\mathcal{M}_{u}^\ep$ is the unique local
unstable manifold of the origin.

\item[4)] $h_{cs}$ is $C^2$ in $W_u$ with norms independent of $\ep$. Moreover,
there exists $C$ independent of $\ep$ such that
\begin{equation}\label{eq4.11}
h_{cs}(0,\ep)=0\ , \ |Dh_{cs}(0,\ep)|\leq C\ep\ , \
|h_{cs}(\cdot,\ep)-h_s^\star(\cdot)|_{C^1}\leq Ce^{\alpha(\ep)}.
\end{equation}
\end{enumerate}
\end{theorem}
\begin{proof}
Part 1) and 3) have been proved in above. The smoothness and $\ep$-independent estimates of $h_{u(cs)}$ also
follow from the standard argument, see \cite{LZ10} for more details. We will focus on \eqref{eq4.10}
and \eqref{eq4.11} can be obtained in a similar way. Let
\[\sigma\triangleq|D\tilde F|_{C^0}+|D(\bar G W^c)|_{C^0},\] which can be taken arbitrarily small by choosing appropriate cut-off functions. First we note from \eqref{eq3.16} that if $W_h=0,W_c=0$, $(0,0)$
is a fixed point of $\mathscr{T}_{cs}^\ep$. By uniqueness, we have
\[h_u(0,0,\ep)=0.\]
For $0<\eta<\frac 12$, we choose $\sigma$ sufficiently small so that
for any $\eta'\in[\eta,2\eta]$,
\begin{equation}\label{eq4.12}
1-\frac {1}{\eta'}(\sigma+e^{\alpha(\ep)}|D\tilde
G|_{C^0})-\frac{1}{1-\eta'}\sigma>\frac 12.
\end{equation}
Given $r>0$ and any $W_{cs}\in B_r(0,X^s\times Y_1)$ we let
$(W^h,W^c)$ be the unique fixed point of
$\mathscr{T}_{cs}^\ep(\cdot,\cdot;W_{cs})$ in $B_{\eta'}^+(\rho)$
for $\eta'\in[\eta,2\eta]$. By using \eqref{eq3.16}, we have
\begin{equation}\label{eq4.13}
|(W^h,W^c)|_{\eta'}^+\leq\frac{r}{1-\frac
{1}{\eta'}(\sigma+e^{\alpha(\ep)}|D\tilde
G|_{C^0})-\frac{1}{1-\eta'}\sigma}\leq 2r.
\end{equation}
Let $(\phi,\psi)$ be the derivative of $(W^h,W^c)$ with respect to
$W_{cs}$ which satisfies
\begin{equation}\label{eq4.14}
\begin{aligned}
\begin{pmatrix} \phi \\
\psi
\end{pmatrix}(\tau)=&\begin{pmatrix}
e^{\tau A}\\
e^{\tau\frac J\ep}
\end{pmatrix}+\int_{+\infty}^\tau\begin{pmatrix} e^{(\tau-s)A}P_u D\tilde
F(\phi,\psi)\\0
\end{pmatrix}\ ds\\
&+\int_0^\tau\begin{pmatrix} e^{(\tau-s)A}P_s D\tilde
F(\phi,\psi)\\
e^{(\tau-s)\frac J\ep}\big(D(\bar GW^c)(\phi,\psi)+e^{\alpha(\ep)}D\tilde G(\phi,\psi)\big)\end{pmatrix}\ ds
\end{aligned}
\end{equation}
which implies for any $\eta'\in[\eta,2\eta]$,
\begin{equation}\label{eq4.15}
|(\phi(\tau),\psi(\tau))|_{\eta'}^+\leq2.
\end{equation}

Fix $W_c=0$ and let $W_\star^h$ be the fixed point of
$\mathscr{T}_s^\star(\cdot;W_s)$. We have
\begin{eqnarray*}
&&(W^h-W_\star^h)(\tau)=\int_0^\tau e^{(\tau-s)A}P_s\big(\tilde
F(W^h,W^c,\ep)-\tilde F(W_\star^h,0,\ep)\big)\ ds\\
&&\hspace{3cm}+\int_{+\infty}^\tau e^{(\tau-s)A}P_u\big(\tilde
F(W^h,W^c,\ep)-\tilde F(W_\star^h,0,\ep)\big)\ ds,\\
&&W^c(\tau)=\int_{+\infty}^\tau e^{(\tau-s)\frac J\ep}
\big(\bar G(W^h,W^c,\ep)+e^{\alpha(\ep)}\tilde G(W^h,W^c,\ep)\big)\ ds.
\end{eqnarray*}
Along with
\eqref{eq3.16} we obtain
\begin{equation}\label{eq4.16}
|(W^h-W_\star^h,W^c)|_{\eta'}^+\leq\frac{e^{\alpha(\ep)}|D\tilde
G|_{C^0}|(W^h,W^c)|_{\eta'}^+}
{1-\frac{\sigma}{\eta'}-\frac{\sigma}{1-\eta'}}\leq
4re^{\alpha(\ep)}|D\tilde G|_{C^0}.
\end{equation}
Consequently,
\[\begin{aligned}|h_u(W_s,0,\ep)-h_u^\star(W_s)|_{X^u}&=|(W^h(0)-W_\star^h(0),W^c(0))|_{X\times Y_1}\\
&\leq |(W^h-W_\star^h,W^c)|_{\eta'}^+ \leq 4re^{\alpha(\ep)}|D\tilde
G|_{C^0}.\end{aligned}\] 

Let $\phi_\star=DW_\star^h(W_s)$ which
satisfies
\[\phi_\star(\tau)=e^{\tau A}+\int_0^\tau e^{(\tau-s)A}P_sD_1\tilde F(W_\star^h,0,\ep)\phi_\star\ ds
+\int_{+\infty}^\tau e^{(\tau-s)A}P_uD_1\tilde
F(W_\star^h,0,\ep)\phi_\star\ ds.\] With slight abuse of notation,
we still use $(\phi,\psi)$ to denote derivative of $(W^h,W^c)$ with
respect to $W_s$ or $W_c$ at $W_c=0$. For the derivative with
respect to $W_s$, $(W^h,W^c)$ satisfies the same equation as in
\eqref{eq4.14} except replacing $e^{\tau\frac J\ep}$ by $0$. Thus,
\[\begin{aligned}
\psi(\tau)=&\int_0^\tau e^{(\tau-s)\frac J\ep}(D_{W^h}\bar GW^c)\phi+D_{W^c}(\bar GW^c)\psi+e^{\alpha(\ep)}D\tilde G(\phi,\psi)\ ds,\end{aligned}\] 
which implies
\begin{equation}\label{eq4.17}
\begin{aligned}|\psi|_{2\eta}^+&\leq
\frac{1}{1-\frac{1}{2\eta}(\sigma+e^{\alpha(\ep)})}\frac{1}{2\eta}\big(|\phi|_\eta^+
4re^{\alpha(\ep)}|D\tilde G|_{C^0}+e^{\alpha(\ep)}|D\tilde
G|_{C^0}|\phi|_{2\eta}^+\big)\\
&\leq \frac{8r+2}{\eta}e^{\alpha(\ep)}|D\tilde G|_{C_0},
\end{aligned}
\end{equation}
where we also use \eqref{eq4.15} and \eqref{eq4.16}. For
$\phi-\phi_\star$, we have
\[\begin{aligned}
(\phi-\phi_\star)(\tau)=&\int_0^\tau e^{(\tau-s)A}P_s\big(D_1\tilde
F(W^h,W^c,\ep)\phi-D_1\tilde F(W_\star^h,0,\ep)\phi_\star+D_2\tilde
F\psi\big)\ ds\\
&+\int_{+\infty}^\tau e^{(\tau-s)A}P_u\big(D_1\tilde
F(W^h,W^c,\ep)\phi-D_1\tilde F(W_\star^h,0,\ep)\phi_\star+D_2\tilde
F\psi\big)\ ds.
\end{aligned}\]
Together with \eqref{eq4.15}, \eqref{eq4.16} and \eqref{eq4.17}, we
obtain
\[|\phi-\phi_\star|_{2\eta}^+\leq (\frac{\sigma}{2\eta}+\frac{\sigma}{1-2\eta})|\phi-\phi_\star|_{2\eta}^+
+Ce^{\alpha(\ep)}|D\tilde G|_{C^0},\] where $C$ depends on
$r,\eta,\sigma,|\phi_\star|_\eta^+,|D^2\tilde F|_{C^0}$.
Consequently, $|\phi-\phi_\star|_{2\eta}^+\leq
2Ce^{\alpha(\ep)}|D\tilde G|_{C^0}$. Therefore,
\begin{equation}\label{eq4.18}
|h_u(\cdot,0,\ep)-h_u^\star(\cdot)|_{C^1}\leq Ce^{\alpha(\ep)},
\end{equation}
where $C$ depends on $r,\eta,\sigma,|\phi_\star|_\eta^+,|D^2\tilde
F|_{C^0}$. For the derivative with respect to $W_c$ at $W_c=0$, we note
\begin{equation}\label{eq4.19}
\psi=\ep J^{-1}\psi_\tau-\ep J^{-1}D(\bar G W^c)
(\phi,\psi)-e^{\alpha(\ep)}D\tilde G(\phi,\psi).
\end{equation}
Plugging the above expression of $\psi$ into $\phi-\phi_\star$, it suffices to prove
\[\begin{aligned}
&\sup_{\tau\geq0}e^{-2\eta\tau}\Big|\int_0^\tau
e^{(\tau-s)A}P_sD_2\tilde F(W_\star^h,0,\ep)\ep J^{-1}\psi_s\ ds\big|\leq C\ep,\\
&\sup_{\tau\geq0}e^{-2\eta\tau}\Big|\int_{+\infty}^\tau
e^{(\tau-s)A}P_uD_2\tilde F(W_\star^h,0,\ep)\ep J^{-1}\psi_s\
ds\big|\leq C\ep.\end{aligned}.\] 
We will only prove the first part. Integrating by
parts to obtain
\begin{eqnarray*}
&&\int_0^\tau e^{(\tau-s)A}P_sD_2\tilde F(W_\star^h,0,\ep)\ep
J^{-1}\psi_s\ ds\\
&=&P_s D_2\tilde F(W_\star^h(\tau),0,\ep)\ep
J^{-1}\psi(\tau)-e^{\tau A}P_sD_2\tilde F(W_\star^h(0),0,\ep)\ep
J^{-1}\psi(0)\\
&&+\int_0^\tau Ae^{(\tau-s)A}P_sD_2\tilde F(W_\star^h,0,\ep)\ep
J^{-1}\psi\ ds\\
&&-\int_0^\tau e^{(\tau-s)A}P_sD_1D_2\tilde
F(W_\star^h,0,\ep)(AW_\star^h+\tilde F(W_\star,0,\ep),\ep
J^{-1}\psi)\ ds.
\end{eqnarray*}
By \eqref{eq4.15}, we have
\begin{equation}\label{eq4.20}
\sup_{\tau\geq0}e^{-2\eta\tau}\Big|\int_0^\tau
e^{(\tau-s)A}P_sD_2\tilde F(W^h,W^c,\ep)\ep J^{-1}\psi_s\ ds\big|
\leq C\ep,
\end{equation}
which completes the proof of part 2).
\end{proof}
We have similar results for center-unstable and stable manifolds.
\begin{theorem}\label{thm4.8}
Assume (H). For the system \eqref{eq3.14}, we have
\begin{enumerate}
\item[1)] There exist $r>0$, $\ep_0>0$ and a mapping $h_s:B_r(0,X^u\times
Y_1)\times(0,\ep_0)\longrightarrow X^s$ such that its graph
$\mathcal{M}_{cu}^\ep$ forms a local center-stable manifold of
the origin.

\item[2)] $h_s$ is $C^2$ in $W_u$ and $W_c$ with norms independent of $\ep$. Moreover,
there exists $C$ independent of $\ep$ such that
\begin{equation}\label{eq4.21}
h_s(0,0,\ep)=0\ , \ |h_s(\cdot,0,\ep)-h_s^\star(\cdot)|_{C^1}\leq Ce^{\alpha(\ep)}\ , \
|D_2h_s(W_s,0,\ep)|\leq C\ep.
\end{equation}

\item[3)] There exist $r>0$, $\ep_0>0$ and a mapping $h_{cu}:B_r(0,X^u)\times(0,\ep_0)\longrightarrow X^u\times Y_1$
such that its graph $\mathcal{M}_{s}^\ep$ is the unique local
unstable manifold of the origin.

\item[4)] $h_{cu}$ is $C^2$ in $W_s$ with norms independent of $\ep$. Moreover,
there exists $C$ independent of $\ep$ such that
\begin{equation}\label{eq4.22}
h_{cu}(0,\ep)=0\ , \ |h_{cu}(\cdot,\ep)-h_u^\star(\cdot)|_{C^1}\leq
Ce^{\alpha(\ep)}.
\end{equation}
\end{enumerate}
\end{theorem}

By taking the intersection of $\mathcal{M}_{cs}^\ep$ and
$\mathcal{M}_{cu}^\ep$, one can obtain a center manifold
$\mathcal{M}_c^\ep$.

\begin{theorem}\label{thm4.9}
Assume (H). There exist $r>0$, $\ep_0>0$ and mappings
$\Psi=(\Psi_s,\Psi_u):B_r(0,Y_1)\times(0,\ep_0)\longrightarrow
X^s\times X^u$ such that its graph $\mathcal{M}_{c}^\ep$ forms a
local center manifold of the origin. Moreover, $\Psi$
is $C^2$ in $W_c$ with norms independent of $\ep$ and
\begin{equation}\label{eq4.25}
\Psi(0,\ep)=0\ , \ |D\Psi(0,\ep)|\leq C\ep.
\end{equation}
\end{theorem}
\begin{proof}
The existence of $\Psi$ is equivalent to find solutions of
\[W_s=h_s(W_u,W_c,\ep)\ , \ W_u=h_u(W_s,W_c,\ep),\]
in terms of $(W_c,\ep)$, which can be solved by the contraction mapping principle.
The estimates on $D\Psi$ can be obtained by differentiating the above equations with respect to $W^c$ and using last inequalities in \eqref{eq4.10} and \eqref{eq4.21}, respectively.
\end{proof}
\section{Intersection of center-stable and center-unstable
manifold}\label{Intersection} In this section, we adopt and modify the method in \cite{SZ03} to deal with our singular system \eqref{eq3.14} to prove the intersection of center-stable and
center-unstable manifold of \eqref{eq3.14}, from which we find breathers with exponentially small tails. The idea is to show the Hamiltonian $H$ is
positive definite on the center manifold and use intermediate value theorem to locate intersection points.\\


\noindent {\bf The Hamiltonian $H$ on the center manifold $\mathcal M_c^\ep$.} Recall that the Hamiltonian $H$ defined \eqref{eq2.5.1a} for \eqref{eq2.5.1} is
\[
\begin{aligned}&H(w^h,w_1^h,w^c,w_1^c,\ep)\\
=&\pi\Big(\frac {(w_1^h)^2}{2}-\frac{(w_h)^2}{2}+\frac{f^{\prime\prime\prime}(0)}{32}(w^h)^4\Big)+\frac{1}{\ep^2}\Big(\int_{-\pi}^\pi \frac{\om^2}{2}((-1-\p_{xx})^{\frac 12}w_1^c)^2\\
&\hspace{1cm}+\frac{\om^2}{2}(w_x^c)^2+\frac{F(\ep w^h\sin x+\ep w^c)}{\ep^2}-\ep^2\frac{f^{\prime\prime\prime}(0)}{24}(w^h\sin x)^4-\frac{(w^c)^2}{2}\ dx\Big),
\end{aligned}\]
where $\om^2=1-\ep^2$. We also recall that we obtain a sequence of symplectic transformations $\Gamma_j$ in Section \ref{Tr}, where $2\leq j\leq[\frac c\ep]$, and a center manifold
\[\mathcal{M}_c^\ep=\{W_c+\Psi(W_c,\ep)\big||W_c|_{Y_1}\leq r,\ep\in(0,\ep_0)\}\]
for \eqref{eq3.14} in Section \ref{IM}. Let 
\[I+\bar\Gamma\triangleq \Big((I+\Gamma_{[\frac c\ep]})\circ\cdots\circ(I+\Gamma_2)\Big)^{-1}.\]
Since $\Gamma_j=O(\frac{\ep^2}{2^k})$, it follows $\bar\Gamma=O(\ep^2)$.
Let
\begin{equation}\label{eq5.2.1}
\tilde H(W^h,W^c,\ep)\triangleq
H((I+\bar\Gamma)\circ(W^h,W^c),\ep).
\end{equation}
Since $H$ is quadratic, we have $\tilde H(0,0,\ep)=0,D\tilde H(0,0,\ep)=0$, which implies
\begin{equation}\label{eq5.2.2}
\tilde H\Big|_{\mathcal M_s^\ep,\mathcal M_u^\ep}=0.
\end{equation}
\begin{lemma}\label{lemma5.1}
There exists $b>0$ and $\ep_0>0$ such that for any
$|W^c|_{Y_1}\leq b,\ep\in(0,\ep_0)$, we have
\begin{equation}\label{eq5.3}
\frac 15\leq\frac{\ep^2 \tilde H(W^c,\ep)}{|W^c|_{Y_1}^2}\leq 1.
\end{equation}
\end{lemma}
\begin{proof}
By TaylorÕs expansion, for sufficiently small $b$ and $\ep$, the leading order of $\ep^2\tilde H$ is given by
\[\frac 12\int_{-\pi}^{\pi}\big((-1-\p_{xx})^{\frac
12}y_1\big)^2+y_x^2-y^2\ dx,\]
which satisfies
\[\frac 15<\frac{3}{10}\leq\frac{\frac 12\int_{-\pi}^{\pi}\big((-1-\p_{xx})^{\frac
12}y_1\big)^2+y_x^2-y^2\ dx}{|y|_{Y_1}^2+|y_1|_{Y_1}^2}\leq\frac12<1.\]
Therefore, the proof is completed.
\end{proof}
Combining the above lemma and \eqref{eq5.2.2}, we have
\begin{equation}\label{eq5.3.1}
\tilde H\Big|_{\mathcal M_{cs}^\ep,\mathcal M_{cu}^\ep}\geq0\ , \ \tilde H\Big|_{\mathcal M_{cs}^\ep\backslash\mathcal M_{s}^\ep,\mathcal M_{cu}^\ep\backslash\mathcal M_u^\ep}>0.
\end{equation}

To study the intersection of center-stable and center-unstable manifolds, we build
up a coordinate system around the unperturbed homoclinic orbit
$h(\tau)$ given in \eqref{eq2.10}.\\

\noindent {\bf Coordinates System near $h$.} First we choose
$x_0\in h$ and let $v(x_0)$ be the vector field of \eqref{eq2.9}
at $x_0$. Let $d\triangleq
DH_0(x_0)$, where 
\[H_0(w^h,w_1^h)=\frac 12((w_1^h)^2-(w^h)^2)+\frac{f^{\prime\prime\prime}(0)}{32}(w^h)^4.\] Since $H_0$ is invariant along any solution of \eqref{eq2.9}, we have $DH_0(x_0)\perp v(x_0)$. Let $\Sigma\triangleq\mbox{Span}\{d\}\oplus Y_1$
and $P_v,P_d,P_{Y_1}$, which are linear projections onto subspaces
$\{\mathbb{R}v\},\{\mathbb{R}d\}$ and $Y_1$, respectively. Note that
a point $p\in\Sigma$ if and only if $P_v(p-x_0)=0$. With slight
abuse of notation, we also use $\mathcal{M}_\ep^\beta$ to denote the
global invariant manifolds extended from the local ones by the flow
map of \eqref{eq3.14}, where $\beta=cs,cu,c,s,u$. Let
$\tilde{\mathcal{M}}_\ep^\beta\triangleq\mathcal{M}_\ep^\beta\cap(X\times
B_{Ce^{\alpha(\ep)}}(0,Y_1))\cap\Sigma$. We claim that
$\tilde{\mathcal{M}}_\ep^\beta$ can be written as local graphs in
the following lemma. We will only present some key points in the
proof and a more detailed presentation can be found in Section 6.1
and 6.3 of \cite{LZ10}.
\begin{lemma}\label{le5.4}
For any $b>0$, there exists $\ep_0>0$ such that for $\tilde
Y=(y,y_1)\in B_b(0,Y_1)$ and $\ep\in[0,\ep_0]$,
$\tilde{\mathcal{M}}_\ep^{cs,cu,s,u}$ contains some local graphs, namely,
\[
\tilde{\mathcal{M}}_\ep^{cs}\supset \{x_0+e^{\alpha(\ep)}\tilde
Y+\Upsilon^d(\tilde Y,\ep)\}\ ,\
\tilde{\mathcal{M}}_\ep^{cu}\supset \{x_0+e^{\alpha(\ep)}\tilde
Y+\Upsilon_1^d(\tilde Y,\ep)\}.\]
Moreover, there exists $C>0$ independent of $\ep$ such that
\begin{equation}\label{eq5.5}
|D\Upsilon^d(\cdot,\ep)|_{C^0(B_b(0,Y_1))}+|D\Upsilon_1^d(\cdot,\ep)|_{C^0(B_b(0,Y_1))}\leq
Ce^{\alpha(\ep)}.
\end{equation}
\end{lemma}
\begin{proof}
Let $\varphi(\tau,\cdot,\ep)$ and $\varphi^\star(\tau,\cdot)$ be the flow
maps of \eqref{eq3.14} and \eqref{eq4.1}, respectively. Fix $W_s\in
B_r(0,X^s)$, there exists $\tau_0$ such that
\[P_v(\varphi^{\star}(\tau_0,W_s+h_u^\star(W_s))-x_0)=0.\]
By Theorem \ref{thm4.3} and \ref{thm4.5}, for arbitrary $b'>0$,
there exists
\[\big(a(\cdot),\tau(\cdot)\big):B_{b'}(0,Y_1)\times[0,\ep_0)\longrightarrow
X^s\times\mathbb{R},\] such that for $\overline Y\in B_{b'}(0,Y_1)$,
\[\begin{aligned}
&P_v\tilde \varphi(\overline Y,\ep)\triangleq
P_v\big(\varphi(\tau(\overline Y,\ep),W_s+a(\overline
Y,\ep)+e^{\alpha(\ep)}\overline Y+h_u(W_s+a(\overline
Y,\ep)+e^{\alpha(\ep)}\overline Y,\ep))-x_0\big)=0.
\end{aligned}\] For any
$b>0$ and $\tilde Y\in B_b(0,Y_1)$, we consider equation
\begin{equation}\label{eq5.6.1}
\frac{1}{e^{\alpha(\ep)}}P_{Y_1}\tilde \varphi(\overline Y,\ep)=\tilde Y.
\end{equation}
Based on Theorem \ref{eq4.5} and \eqref{eq4.10}, we have for $\delta
Y\in Y_1$,
\[
\frac{1}{e^{\alpha(\ep)}}\big|DP_{Y_1}\tilde \varphi(\overline
Y,\ep)\delta Y -E(t_0,0)e^{\alpha(\ep)}\delta
Y\big|\leq C\ep,
\] where $E$ is the evolutionary
operator generated by \eqref{eq4.6} in the normal direction. Since
$E$ is invertible with norm independent of $\ep$, \eqref{eq5.6.1}
has a unique solution for each $\tilde Y\in B_b(0,Y_1)$. Now we
define
\[\Upsilon^d(\tilde Y,\ep)=P_d\tilde \varphi((\frac{1}{e^{\alpha(\ep)}}P_{Y_1}\tilde \varphi)^{-1}\tilde Y,\ep),\]
and $\Upsilon^d_1$ can be obtained in a similar way. Clearly,
\eqref{eq5.5} follows from the rescaling.
\end{proof}

For each $\ep$, since $P_{Y_1}(\tilde {\mathcal M}_\ep^s-x_0)$ and $P_{Y_1}(\tilde {\mathcal M}_\ep^u-x_0)$ are
$O(e^{\alpha(\ep)})$, the following quantities
\begin{equation}\label{eq5.6.2}
\tilde Y(\ep)=\frac{1}{e^{\alpha(\ep)}}(P_{Y_1}\tilde {\mathcal M}_\ep^s-x_0)\ , \
\tilde
Y_1(\ep)=\frac{1}{e^{\alpha(\ep)}}(P_{Y_1}\tilde {\mathcal M}_\ep^u-x_0),\end{equation} are
well defined. Consequently,
\[\tilde{\mathcal{M}}_\ep^s=\{x_0+\tilde Y(\ep)+\Upsilon^d(\tilde Y(\ep),\ep)\}\ , \
\tilde{\mathcal{M}}_\ep^u=\{x_0+\tilde Y_1(\ep)+\Upsilon^d(\tilde
Y_1(\ep),\ep)\}.\]
\\
\noindent{\bf The intersection of $\tilde {\mathcal{M}}_\ep^{cs}$
and $\tilde {\mathcal{M}}_\ep^{cu}$.} For $s\in[0,1]$ and $\tilde Y(\ep),\tilde Y_1(\ep)\in B_b(0,Y_1)$ given in \eqref{eq5.6.2}, we let
\begin{eqnarray*}
&& q^\ep(s)=s\tilde Y(\ep)+(1-s)\tilde Y_1(\ep),\\
&&
p^\ep(s)=x_0+e^{\alpha(\ep)}q^\ep(s)+\Upsilon^d(q^\ep(s),\ep)\in\tilde{\mathcal{M}}_\ep^{cs},\\
&&p_1^\ep(s)=x_0+e^{\alpha(\ep)}q^\ep(s)+\Upsilon_1^d(q^\ep(s),\ep)\in\tilde{\mathcal{M}}_\ep^{cu},\\
&& h(s)=\tilde H(p^\ep(s),\ep)-\tilde H(p_1^\ep(s),\ep),
\end{eqnarray*}
where $\tilde H$ is defined in \eqref{eq5.2.1}.

The intersection of center-stable and center-unstable
manifold is given by $\Upsilon^d=\Upsilon_1^d$, which is equivalent to $h(s)=0$. This is because the leading order of $D\tilde H(\mathcal {\tilde M}_\ep^{cs(cu)})$ is given by $DH_0(x_0)$. By \eqref{eq5.3.1},
\[\begin{aligned}h(1)=&\tilde H(x_0+e^{\alpha(\ep)}\tilde Y(\ep)+\Upsilon^d(\tilde Y(\ep),\ep))-\tilde H(x_0+e^{\alpha(\ep)}\tilde Y(\ep)+\Upsilon_1^d(\tilde Y(\ep),\ep))\\
\leq&0\leq  \tilde H(x_0+e^{\alpha(\ep)}\tilde Y_1(\ep)+\Upsilon^d(\tilde Y_1(\ep),\ep))-\tilde H(x_0+e^{\alpha(\ep)}\tilde Y_1(\ep)+\Upsilon_1^d(\tilde Y_1(\ep),\ep))= h(0).\end{aligned}\] 
The intermediate value theorem implies $h(s_0)=0$
for some $s_0\in[0,1]$. Finally, it is easy to see $\tilde
H(p^\ep(s_0),\ep)\leq Ce^{\alpha(\ep)}$, which implies the tail of the perturbed breather is exponentially small in $\ep$. We summarize results from above analysis in the following theorem.
\begin{theorem}\label{thm5.2}
Assume (H). There exists $\ep_0>0$ such that for any
$\ep\in(0,\ep)$, $\mathcal{M}_\ep^{cs}$ and $\mathcal{M}_\ep^{cu}$
of \eqref{eq3.14} have a nonempty intersection. The
solutions lying on the intersection converge both forward and
backward in time to fast oscillatory solutions with exponentially small amplitudes on the center
manifold.
\end{theorem}

Finally, by taking account of all recalings and switching $x$ and $t$ back, we complete the proof of the Main Theorem.
\section*{Acknowledgement}
The author would like to thank Professor Chongchun Zeng for many useful discussions during the preparation of this paper.

\end{document}